\documentclass{proc-l}

\usepackage{amssymb}

\usepackage[cmtip,all]{xy}

\usepackage{amsthm}
\usepackage{amsgen}
\usepackage{amsmath}
\usepackage{epic}
\usepackage[dvips]{graphicx}
\usepackage{psfrag}

\newtheorem{theorem}{Theorem}[section]

\newtheorem{thm}{Theorem}[section]
\newtheorem{lem}[thm]{Lemma}
\newtheorem{question}[thm]{Question}

\theoremstyle{definition}

\theoremstyle{remark}

\numberwithin{equation}{section}

\newcommand{\Real}{{\mathbb R}}

\newcommand{\Rat}{{\mathbb Q }}
\newcommand{\Zed}{{\mathbb Z}}

\newcommand{\Diff}{{\mathrm{Diff}}}

\newcommand{\lar}{{\Lambda}}
\newcommand{\fof}{{\mathbb Q(\Lambda)}}

\begin{document}

\title{An obstruction to a knot being deform-spun via Alexander polynomials}


\author{Ryan Budney }
\address{Mathematics and Statistics, University of Victoria \\
PO BOX 3045 STN CSC, Victoria, B.C., Canada V8W 3P4}
\email{rybu@uvic.ca,}
\thanks{Both authors would like to thank the Max Planck Institute for Mathematics for
its hospitality. The first author would also like to thank the 
Institut des Hautes \'Etudes Scientifiques for its hospitality, as well as
Danny Ruberman and an anonymous referee for many useful comments on the paper.}
\author{Alexandra Mozgova}
\address{ACRI, 260 route du Pin Montard, BP 234 \\
F-06904 sophia-Antipolis Cedex - France}
\email{sasha.mozgova@gmail.com}


\date{December 9th, 2007}

\dedicatory{}

\commby{}

\begin{abstract}
We show that if a co-dimension two knot is deform-spun from a lower-dimensional co-dimension 2 knot, there are constraints on the Alexander polynomials.  In particular this shows, for all $n$, that not 
all co-dimension 2 knots in $S^n$ are deform-spun from knots in $S^{n-1}$.   
\end{abstract}

\maketitle


In co-dimension 2 knot theory \cite{Kaw}, typically the term  `$n$-knot' denotes a manifold pair
$(S^{n+2},K)$ where $K$ is the image of a smooth embedding $f : S^n \to S^{n+2}$. 
An $n$-ball pair is a pair $(D^{n+2},J)$ where $J$ is the image of a smooth
embedding $f : D^n \to D^{n+2}$ such that $f^{-1}(\partial D^{n+2})=\partial D^n$.
Every $n$-knot $K$ is isotopic to a union 
$(S^{n+2},K) = (D^{n+2},J) \cup_{\partial} (D^{n+2},D^n)$ for some unique isotopy
class of $n$-ball pair $(D^{n+2},J)$ provided we consider $K$ to be oriented. 
Let $\Diff(D^{n+2},J)$ denote the group of diffeomorphisms of an $n$-ball pair $(D^{n+2},J)$. 
That is, $f \in \Diff(D^{n+2},J)$ means that $f$ is a diffeomorphism of $D^{n+2}$ which
restricts to the identity on $\partial D^{n+2} = S^{n+1}$, is isotopic to the identity (rel boundary)
as a diffeomorphism of $D^{n+1}$, and $f$ preserves $J$, $f(J)=J$.  We say an $n$-knot 
$(S^{n+2},K)$ is deform-spun from an $(n-1)$-knot 
$(S^{n+1},K') = (D^{n+1},J') \cup_\partial (D^{n+1},D^{n-1})$ if there exists
$g \in \Diff(D^{n+1},J')$ such that the pair 
$\left( (D^{n+1},J') \times_g S^1 \right) \cup_\partial \left( (S^{n},S^{n-1}) \times D^2 \right)$
 is diffeomorphic to the pair $(S^{n+2},K)$.  Here $(D^{n+1},J') \times_g S^1$ is the bundle over
$S^1$ with fibre $(D^{n+1},J')$ and monodromy given by $g$, ie: $(D^{n+1},J') \times_g S^1 = ((D^{n+1},J') \times \Real) / \Zed$ where $\Zed$ acts diagonally, by $g$ on $(D^{n+1},J')$ and as the group of universal covering transformations for $\Real \to S^1$. 

\begin{figure}[ht]\label{spin}
\psfrag{t1}[tl][tl][1][0]{$\{t\}\times int(D^{n+1})$}
\psfrag{t2}[tl][tl][1][0]{$S^n$}
$$\includegraphics[width=9cm]{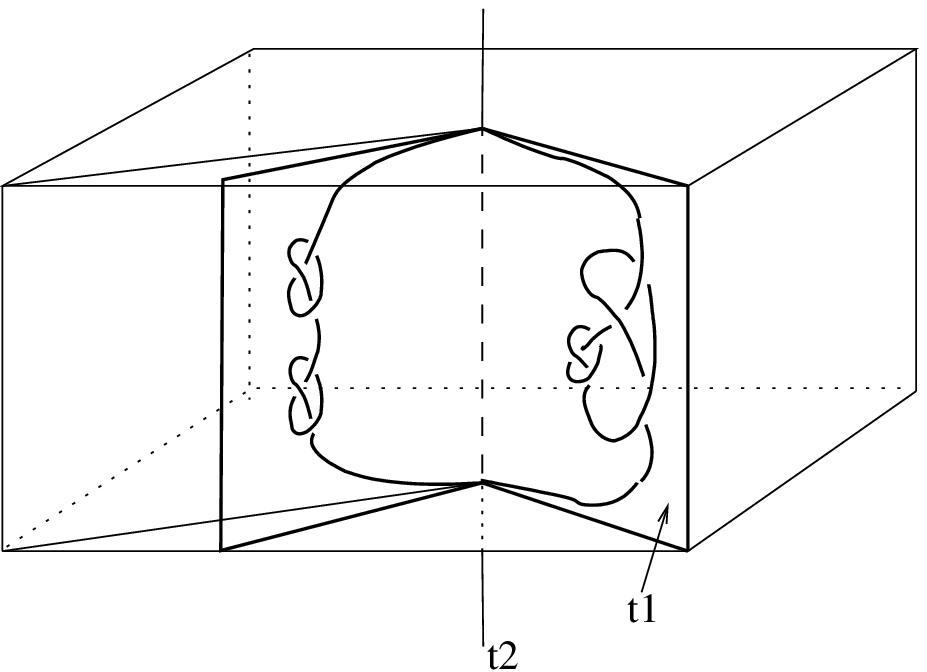}$$
\centerline{\small Figure 1}
\end{figure}

To picture a deform-spun knot, let $g_t$ be a null-isotopy of $g$, ie:
$g_0 = g$, $g_1 = Id_{D^{n+1}}$ and $g_t$ is a diffeomorphism
of $D^{n+1}$ which restricts to the identity on $\partial D^{n+2}$ for all
$0 \leq t \leq 1$. Consider $S^{n+2}$ to be the union of a great
$n$-sphere $S^n$ and a disjoint trivial vector bundle over $S^1$.  Identify
this trivial vector bundle over $S^1$ with $S^1 \times int(D^{n+1})$, and
identify $S^1$ with $\Real/\Zed$. We assume that the inclusion
$S^1 \times int(D^{n+1}) \to S^{n+2}$ extends to a map 
$S^1 \times D^{n+1} \to S^{n+2}$ such that
the restriction $S^1 \times S^n \to S^{n+2}$ factors as projection onto
the great sphere $S^n$ followed by inclusion $S^n \to S^{n+2}$.  
Then the set $\{ (t,x) \in S^1 \times int(D^{n+1}) :
 x = g_t(p), p \in int(J') \}$ is a subset of $S^{n+2}$ whose closure is an
$n$-knot. This is the deform-spun knot, see Figure 1.

The main observation of this paper is that if $K$ is an $n$-knot, 
deform-spun from an $(n-1)$-knot $K'$, then there is a relationship between
the Alexander modules of $K$ and $K'$ which give rise to 
constraints on the Alexander polynomials $\Delta_1, \cdots, \Delta_n$ of $K$.  

\begin{thm}\label{mainthm}Let $K$ be a $n$-knot which is deform-spun,
then there exist polynomials $q_i \in \lar = \Rat[t^{\pm 1}] = \Rat[\Zed]$ 
for $i=0,1,\cdots,n$ which satisfy $q_{i+1}q_i = \Delta_{i+1}$ ($q_0 = q_n = 1$) and 
$q_{n-i} = \overline{q_i}$
for all $i$, where we use the convention $\overline{q_i}(t)=q_i(t^{-1})$.  
\end{thm}

An elementary consequence of this theorem is that for each $n \geq 2$, not every $n$-knot 
is deform-spun from an $(n-1)$-knot.  This follows from the work of Levine \cite{levine} who gave a characterization of the Alexander modules of co-dimension $2$ knots.  In particular Levine 
shows that an $n$
knot has Alexander polynomials $\Delta_1, \cdots, \Delta_n \in \lar$ which satisfy the relations
$\Delta_i(1) \neq 0$, $\overline{\Delta_i} = \Delta_{n-i}$ for all $i$. Moreover, these relations
are complete in the sense that given any $n$ polynomials which satisfy these relations, there
is an $n$-knot which has the specified Alexander polynomials.  
The case $n=2$ has a particularly simple example.  Theorem \ref{mainthm} states that if 
$K$ is deform-spun, then $\overline{\Delta_1}=\Delta_1$, yet there are $2$-knots such 
that $\Delta_1$ is not symmetric. See example 10 of Fox's Quick Trip \cite{fox}, which describes a $2$-knot such that $\Delta_1(t) = 2t-1$.

Litherland's deform-spinning construction has its origin in papers of Fox and
Zeeman.  Fox's `Rolling' \cite{roll} paper gave a heuristic outline of the notion 
eventually called deform-spinning, as a graphing process from a `relative 2-dimensional
braid group' which nowadays is frequently called the fundamental group of the space of knots, 
or (in a slightly different setting) the mapping class group of the knot complement 
\cite{BudFam}.  Zeeman proved that the complements of co-dimension two $n$-twist-spun
knots fibre over $S^1$ provided $n \neq 0$ \cite{Zeeman}.  
Litherland \cite{Lith} went on to formulate a general situation where deform-spun knot 
complements fibre over $S^1$.  Specifically, Litherland proved that if the diffeomorphism 
$g : (D^{n+1},J') \to (D^{n+1},J')$
preserves a Seifert surface for the knot $(S^{n+1},K')$ corresponding to the $(n-1)$-disc pair
$(D^{n+1},J')$, then the deform-spun knot associated
to the diffeomorphism $M \circ g : (D^{n+1},J') \to (D^{n+1},J')$ has a complement which 
fibres over $S^1$, provided $M : (D^{n+1},J') \to (D^{n+1},J')$ is a non-zero power of the meridional Dehn twist about $J'$.

This paper was largely motivated by a result in `high' co-dimension knot theory. 
In the paper \cite{BudFam} the first author gave a new proof of Haefliger's theorem, 
that the monoid of isotopy classes of smooth embeddings of $S^j$ in $S^n$ is a group, provided
$n-j>2$.  The heart of the proof is showing that if $n-j>2$ then every knot $(S^n,K)$ (where
$K \simeq S^j$) is deform-spun from a lower-dimensional knot $(S^{n-1},K')$, where $K' \simeq S^{j-1}$.  Moreover, all knots $(S^n,K)$ are $i$-fold deform-spun for $i = 2(n-j)-4$, in the sense that one 
obtains $(S^n,K)$ be iterating the deform-spinning process $i$ times. So in a sense this paper represents an investigation of the extreme case $n-j=2$.  A second motivation is
the observation that frequently the groups $\pi_0 \Diff(D^3,J')$ ($(D^3,J')$ a $1$-ball
pair) are quite large \cite{BudFam}, in the sense that their classifying spaces all have the
homotopy-type of finite-dimensional manifolds, but the dimension of these manifolds can be 
arbitrarily large. So there are many ways to construct $2$-knots by deform-spinning a $1$-knot. As far as the authors know, this paper represents the first known
obstructions to knots being deform-spun.

\section{Asymmetry obstruction}

Given a co-dimension $2$ knot $K$ in $S^{n+2}$, the complement of the knot, 
$C_K$ is a homology $S^1$. Let $\tilde C_K$ denote the universal abelian
cover of $C_K$, ie: the cover corresponding to the kernel of the abelianization map
$\pi_1 C_K \to \Zed$, and consider $H_i (\tilde C_K;\Rat)$ to be a module over
the group-ring of covering transformations $\lar = \Rat[\Zed] = \Rat[t,t^{-1}]$, this
is called the $i$-th Alexander module of $K$.   
$H_i (\tilde C_K;\Rat)$ is a finitely-generated torsion $\lar$-module \cite{levine} for
each $i$, so $H_i (\tilde C_K;\Rat) \simeq \bigoplus_j \lar/p_j$ for some collection 
of polynomials $p_j$.  The product of these polynomials $\prod_j p_j$ is called the 
$i$-th Alexander polynomial of $K$, or the order ideal of the $i$-th Alexander module
 $H_i(\tilde C_K;\Rat)$, denoted $\Delta_i$. In general, the order ideal of a finitely
generated torsion $\lar$-module $M$ will be denoted $\Delta_M$.   
A theorem of Levine's \cite{levine} is
that Poincar\'e Duality combined with the Universal Coefficient Theorem induces an isomorphism 
$\overline{H_i(\tilde C_K;\Rat)} \simeq Ext_\lar(H_{n+1-i} (\tilde C_K;\Rat), \lar)$. 
Here, if $M$ is a $\lar$-module, $\overline{M}$ denotes the conjugate $\lar$-module. 
This is a module whose underlying $\Rat$-vector space is $M$, but where action of the 
generator $t$ on $\overline{M}$ is defined as the action of $t^{-1}$ on $M$. 
Thus, the only Alexander polynomials
of $K$ which can be non-trivial are $\Delta_1, \cdots, \Delta_n$, and they satisfy the
relation $\overline{\Delta_i} = \Delta_{n+1-i}$ for all $i$.

We collect some elementary results about $\lar$-modules that will be of use in the proof
of Theorem \ref{mainthm}.  To state the lemma, let $\fof$ denote the field of fractions of
$\lar$, ie: the field which consists of rational Laurent polynomials. 

\begin{lem}\label{lem1}
\begin{itemize}
\item[(a)] (see \cite{Kaw} 7.2.7) Given a short exact sequence 
of finitely generated torsion $\lar$-modules $$ 0 \to H_1 \to H \to H_2 \to 0 $$
the order ideals satisfy $\Delta_{H_1} \Delta_{H_2} = \Delta_H$.
\item[(b)] (see \cite{levine} Proposition 4.1) Let $H$ be a finitely-generated torsion $\lar$-module. 
There is a natural isomorphism of $\lar$-modules
$$ Ext_{\lar}(H,\lar) \simeq Hom_{\lar}(H,\fof/\lar).$$
\item[(c)] With the same setup as (b), there is a natural
isomorphism of $\Rat$-vector spaces
$$Hom_{\lar}(H,\fof/\lar) \simeq Hom_{\Rat}(H,\Rat)$$
where we interpret $\lar \subset \fof$ as the rational Laurent polynomials
with denominator $1$. 
\item[(d)] Let $g : H \to H$ be a $\lar$-linear map, where $H$ is a finitely-generated
torsion $\lar$-module. 
Let $g^* : Ext_{\lar}(H,\lar) \to Ext_{\lar}(H,\lar)$ the Ext-dual of $g$. Then $ker(g)$ and $ker(g^*)$ have the same order ideals.
\end{itemize}
\end{lem}

\begin{proof}(of item (c))
Consider a rational polynomial $\frac{p}{q} \in \fof$. The division 
algorithm allows us to write $p = sq + r$ for Laurent polynomials $s, r \in \lar$
where $r \in \Rat[t]$ and $deg(r) < deg(q)$.  To ensure that $r$ is unique, 
we demand that $GCD(p,q)=1$, $q \in \Rat[t]$ and the constant coefficient of
$q$ is $1$.  Define a function $\fof/\lar \to \Rat$ 
by sending $\frac{p}{q}$ to the constant coefficient of $r$. Composition with this
map is a $\Rat$-linear homomorphism $Hom_{\lar}(H,\fof/\lar) \to Hom_\Rat(H,\Rat)$
which is natural and respects connect-sum decompositions of the domain $H$. 
Thus to verify that it is an isomorphism, we need to only check it on a torsion 
$\lar$-module with one generator.
$$Hom_{\lar}(\lar/p,\fof/\lar) \to Hom_\Rat(\lar/p,\Rat)$$
In this case the target space has dimension $deg(p)$; the basis given by the dual basis
to the polynomials $t^i$ for $0 \leq i < deg(p)$. The domain also has dimension $deg(p)$, 
with basis given by homomorphisms that send $1$ to $t^i/p$ where 
$0 \leq i < deg(p)$. Hence the map is a bijection between these basis vectors. 

To prove item (d), consider the `prime factorization' of $H$. 
Let $P \subset \lar$ be the prime factors of the order ideal $\Delta_H$. 
Given $p \in P$ let $H_p \subset H$ be the sub-module
of elements of $H$ killed by a power of $p$, thus $\bigoplus_{p \in P} H_p \simeq H$.
$g$ must respect the splitting, so we have maps $g_p$ such that:  
$$g = \bigoplus_{p \in P} g_p : H_p \to H_p.$$
Thus, 
$$\Delta_{ker(g)} = \prod_{p \in P} \Delta_{ker(g_p)}.$$
Let $d_p \in \Zed$ be defined so that 
$\Delta_{ker(g_p)} = p^{d_p}$.
By part (c), $g$ and $g^*$ can be thought of as the
$Hom_\Rat(\cdot,\Rat)$-duals of each other, thus $ker(g)$ and $ker(g^*)$ have
the same dimension as $\Rat$-vector spaces, and so 
$dim_{\Rat}(ker(g_p)) = deg(p)d_p$, and
$\Delta_{ker(g_p)}$ is determined by the rank of
$ker(g_p)$ as a $\Rat$-vector space.  Hence $ker(g)$ and $ker(g^*)$ have the same 
order ideals.
\end{proof}

{\it Remark.} Although they have the same order ideals, 
in general the two kernels are not isomorphic as $\lar$-modules.
An example is given by 
$g : \lar/p \oplus \lar/p^2 \to \lar/p \oplus \lar/p^2$
defined by $g(a,b) = (0,pa)$. In this case, 
$ker(g) \simeq \lar/p^2$, while
$ker(g^*) \simeq \bigoplus_2 \lar/p$.

\begin{proof}({\it of Theorem \ref{mainthm}})
Let $C_K$ be the complement of an open tubular neighbourhood of $K \subset S^{n+2}$, and 
$C_{K'}$ the complement of an open tubular neighbourhood of $K' \subset S^{n+1}$.
As in the introduction, let $g : (D^{n+1}, J') \to (D^{n+1},J')$ be the 
diffeomorphism for the deform-spinning construction of $K$ from $K'$, so
we can isotope $g$ so that it preserves a regular neighbourhood of $J' \cup S^n$, 
therefore $g$ restricts to a diffeomorphism of $C_{K'}$ (which we can think of as the
complement of an open regular neighbourhood of $S^n \cup J'$ in $D^{n+1}$), giving
a diffeomorphism
$$C_K \simeq (C_{K'} \times_g S^1) \cup_{\nu S^1 \times S^1} ((\nu S^1) \times D^2).$$
where $\nu S^1$ is a trivial $D^{n-1}$-bundle over $S^1$ (a meridian of $\partial C_{K'}$).
The decomposition lifts to the universal abelian covering space, giving the
isomorphism $H_1(\tilde C_K;\Rat) \simeq coker(I-g_{1*})$ and short exact sequences
$$ 0 \to coker(g_{i*} - I) \to H_i(\tilde C_K;\Rat) \to ker(g_{(i-1)*}-I) \to 0, \ \ i > 1$$
with $g_{i*} : H_i (\tilde C_{K'};\Rat) \to H_i(\tilde C_{K'};\Rat)$ the induced map
coming from $\tilde g : \tilde C_{K'} \to \tilde C_{K'}$.  
Let $q_i$ be the order ideal of $coker(g_{i*}-I)$.

The map $g_{i*} - I : H_i(\tilde C_{K'};\Rat) \to H_i(\tilde C_{K'};\Rat)$ give
rise to a canonical short exact sequence
$$ 0 \to ker(g_{i*}-I) \to H_i(\tilde C_{K'};\Rat) \to img(g_{i*}-I) \to 0$$
and the inclusion $img(g_{i*}-I) \to H_i(\tilde C_{K'};\Rat)$ to another
$$ 0 \to img(g_{i*}-I) \to H_i(\tilde C_{K'};\Rat) \to coker(g_{i*}-I) \to 0.$$
Lemma \ref{lem1} (a) applied to our short exact sequences tells us
that $\Delta_i = q_i q_{i-1}$. 

We now reconsider the proof of the symmetry of the Alexander polynomial of a knot in $S^3$
\cite{GordSurv, Kaw}, or more precisely, the isomorphism 
$\overline{H_i(\tilde C_{K'};\Rat)} \simeq H_{n-i}(\tilde C_{K'};\Rat)$
derived from Poincar\'e Duality \cite{levine},
paying special attention to naturality with respect to diffeomorphisms $g \in \Diff(C_{K'})$, 
with an eye towards proving the symmetry conditions $\overline{q_{n-i}}=q_i$.
\begin{enumerate}
\item $H_i (\tilde C_{K'};\Rat) \simeq H_i(\tilde C_{K'},\partial; \Rat)$: this is a natural isomorphism coming from the long exact sequence of a pair.
\item $H_i(\tilde C_{K'},\partial;\Rat) \simeq \overline{H^{n+1-i}(\tilde C_{K'};\Rat)}$: this is the Poincar\'e duality isomorphism; it is also natural, although it reverses arrows \cite{levine}. 
\item $H^{n+1-i}(\tilde C_{K'};\Rat) \simeq Ext_\lar(H_{n-i}(\tilde C_{K'};\Rat),\lar)$: this is 
a natural isomorphism coming from the universal coefficient theorem \cite{levine}.
\item $Ext_\lar(H_{n-i}(\tilde C_{K'};\Rat),\lar) \simeq H_{n-i}(\tilde C_{K'};\Rat)$. This last
result uses that both modules have a square presentation matrix, with one
being the transpose of the other.  Since $\lar$ is a principal ideal domain, 
the presentation matrices are equivalent to the same diagonal matrices. 
This isomorphism is not natural.
\end{enumerate}
Thus we have a non-natural isomorphism $H_i(\tilde C_K;\Rat) \simeq \overline{H_{n-i}(\tilde C_K;\Rat)}$. The natural part of the isomorphism can be expressed by the
commutative diagram

$$ \xymatrix{
\overline{H_i(\tilde C_K)} \ar[r] \ar[d]^{g_*} & \overline{H_i(\tilde C_K,\partial)}
\ar[r]^-{PD} \ar[d]^{g_*} & H^{n+1-i}(\tilde C_K) & Ext_{\lar}\left(H_{n-i}(\tilde C_K),\lar\right) \ar[l]_-{UCT} \\
\overline{H_i(\tilde C_K)} \ar[r] & \overline{H_i(\tilde C_K,\partial)} \ar[r]^-{PD} & 
 H^{n+1-i}(\tilde C_K) \ar[u]_{g^*} & Ext_{\lar}\left( H_{n-i}(\tilde C_K),\lar\right)
 \ar[u]_{(g_*)^*} \ar[l]_-{UCT} } $$

This gives us an isomorphism of $\lar$-modules 
$\overline{ker(I-g_{i*})} \simeq ker(I - (g_{(n-i)*}^{-1})^*)$, so
$$\overline{ker(I-g_{i*})} \simeq ker(I - (g_{(n-i)*}^{-1})^*) = ker(I - (g_{(n-i)*})^*).$$ 
Lemma \ref{lem1} (d), tells us that $ker(I-(g_{(n-i)*})^*)$ and $ker(I-g_{(n-i)*})$ have the same
order ideals.  Thus, $\overline{q_i} = q_{n-i}$. 
\end{proof}

\section{Comments and questions}

Levine \cite{levine} has a complete characterization of the Alexander modules of co-dimension
two knots. A natural question would be, could one derive further other obstructions to deform-spinning
from the Alexander modules of knots?  The primary aspect of Levine's work that we've neglected
is the $\Zed$-torsion submodule of $H_i(\tilde C_K;\Zed)$.  Simple experiments show that when
$K \subset S^{n+2}$ is deform-spun from a knot $K' \subset S^{n+1}$, the Alexander modules
of $K$ can have $\Zed$-torsion, even when the Alexander modules of $K'$ do not.  Moreover, 
twist-spinning sufficies to produce many such examples. So any torsion obstructions to 
deform-spinning, if they exist, would likely be fairly subtle. 

In co-dimension larger than two, deform-spinning is the boundary map in the pseudo-isotopy
long exact sequence for embedding spaces and diffeomorphism groups \cite{BudFam}.  Moreover, 
Cerf's Pseudoisotopy Theorem states that, in the case of diffeomorphism groups of discs, 
this map is onto, provided the dimension of the disc is $6$ or larger.  So one might
expect an analogy. 

\begin{question}
Is there a simple characterization of deform-spun co-dimension two knots $K \subset S^{n+2}$
(provided $n$ is large)? 
\end{question}

One would certainly expect more obstructions to deform-spinning than the ones in this
paper.  For example, let $K_1$ and $K_2$ be two otherwise unrelated $2$-knots such that 
$\Delta_{K_1}(t) = 2-t$ and $\Delta_{K_2}(t) = 2t-1$. Their connect sum has Alexander
polynomial $\Delta_{K_1 \# K_2}(t) = -2t^2 + 3t -2$ which is symmetric, but we have no 
reason to expect $K_1 \# K_2$ is deform-spun.

\bibliographystyle{amsplain}
\providecommand{\bysame}{\leavevmode\hbox to3em{\hrulefill}\thinspace}

\end{document}